\documentclass[11pt]{amsart}
\usepackage{amsfonts}
\usepackage{amsmath}
\usepackage{amssymb}
\usepackage{graphicx}
\usepackage{amscd}
\usepackage{calc}
\usepackage{amsthm}

\newtheorem{theorem}{Theorem}[section]

\newtheorem{corollary}{Corollary}[section]

\newtheorem{lemma}{Lemma}[section]

\numberwithin{equation}{section}
\theoremstyle{remark}
\newtheorem{rmk}{Remark}[section]

\renewcommand{\bar}{\overline}

\newcommand{\pa}{\partial}
\renewcommand{\phi}{\varphi}
\newcommand{\wt}{\widetilde}

\newcommand{\ka}{K\"ahler }

\newcommand{\he}{Hermitian-Einstein }

\newcommand{\R}{{\mathbb R}}

\newcommand{\M}{{\mathcal M}}

\newcommand{\ke}{K\"ahler-Einstein }

\newcommand{\hh}{{\mathcal H}}

\newcommand{\lb}{\left (}
\newcommand{\rb}{\right )}
\newcommand{\lsb}{\left [}
\newcommand{\rsb}{\right ]}
\newcommand{\lfb}{\left \{}
\newcommand{\rfb}{\right \}}

\newcommand{\ga}{\alpha}

\newcommand{\gm}{\gamma}
\newcommand{\gd}{\delta}

\newcommand{\ii}{\sqrt{-1}}

\newcommand{\tr}{\text{Tr}}

\newcommand{\ud}{\underline}
\newcommand{\End}{\text{End}}

\title{The  Hermitian-Yang-Mills Iteration on Stable Bundles}

\author{Huai-Dong Cao$^{\dagger}$, Xiaofeng Sun, and Yingying Zhang}
\address{Department of Mathematics,  Lehigh University,
Bethlehem, PA 18015, USA}
\email{huc2@lehigh.edu; xis205@lehigh.edu}
\address{
Beijing Institute of Mathematical Sciences and Applications, Beijing, 101408,
China. }
\email{zhangyingying@bimsa.cn}

\thanks{$^{\dagger}$Research supported in part by a Simons Fellowship and a grant from the Institute for Advanced Study School of Mathematics during the Spring 2026 term.}

\begin{document}

\begin{abstract}
In this paper, based on recent results for the prescribed Hermitian-Yang-Mills (HYM) tensor \cite{WYY1} and its twisted variants \cite{WYY3}, we provide a dynamical construction of Hermitian-Einstein metrics on stable holomorphic vector bundles and its extension to Higgs bundles. Additionally, in the appendix, we use the heat flow method to give a new proof of the existence and uniqueness of solutions to the twisted prescribed HYM tensor equation from \cite{WYY3}, as well as its generalization to Higgs bundles.
\end{abstract}

\maketitle


\section{Introduction}\label{s10}

 Let $E\to X$ be a holomorphic vector bundle of rank $r$ over a compact K\"ahler manifold $(X,\omega)$ of complex dimension $n$. Given a Hermitian metric $h$ on $E$, the curvature $R^h$ of its Chern connection is an $\End(E)$-valued $(1,1)$-form.
The {\it slope} of $E $ is defined as
\[
\mu(E)=\frac{\deg(E)}{\operatorname{rk}(E)},
\]
where $\deg(E)$ is the {\it degree} of $E$ given by
\[
\deg(E) = \int_X c_1(E)\wedge
\frac{\omega^{n-1}}{(n-1)!} =\frac{1}{\pi} \int_X \tr (R^h)\wedge\frac{\omega^{n-1}}{(n-1)!}
\]
and this number is independent of the chosen Hermitian metric $h$.

$E$ is said to be {\it stable} if, for every proper coherent subsheaf $\mathcal{F}$ of $E$,
\[
\mu(\mathcal{F})<\mu(E).
\]
Also, a Hermitian metric $h$ on $E$ is called a {\it Hermitian-Einstein metric} (or its Chern connection is an {\it Hermitian-Yang-Mills connection}) if
\begin{equation}\label{5}
\Lambda_\omega R^h =\mu \, I_E
\end{equation}
for some real constant $\mu$. Here, $\Lambda_\omega$ denotes contraction with the K\"ahler form $\omega$ and $I_E\in \End(E)$ is the identity.

The fundamental link between algebraic geometry (stable vector bundles) and differential geometry (Hermitian-Einstein metrics) is provided by the celebrated Donaldson-Uhlenbeck-Yau theorem \cite{sd1, sd2, UY}, which states that $E$ admits a Hermitian-Einstein metric if and only if $E$ is (poly)stable.  Over the years, the Donaldson-Uhlenbeck-Yau theorem has been generalized to holomorphic vector bundles over the non-K\"ahler setting by Li and Yau \cite{LY} and to Higgs bundles by Simpson \cite{sim1988}.

Very recently, Wang-Yang-Yau \cite{WYY3} solved the twisted prescribed Hermitian-Yang-Mills tensor problem. Precisely,  they proved the following result.

\begin{theorem}\textup{\bf (Wang-Yang-Yau \cite{WYY3})}\label{05}
Let $E$ be a holomorphic vector bundle over a compact K\"ahler manifold $(X,\omega)$. Then,
\begin{eqnarray}\label{08}
\inf_{\mathcal F}\mu\lb E/\mathcal F\rb=\sup_{\Lambda_\omega R^h>\lambda I_E}\lambda,
\end{eqnarray}
where the infimum is taken over all proper saturated torsion-free subsheaves of $E$ and the supremum is taken over all Hermitian metrics on $E$. This number is denoted by $\lambda_{\min}$.

Moreover,  for any $\lambda<\lambda_{\min}$ and any Hermitian metric $\xi$ on $E$, the equation
\begin{eqnarray}\label{09}
  \Lambda_\omega R^h=\lambda I_E+\xi h^{-1}
\end{eqnarray}
has a unique solution.
\end{theorem}

It is clear that identity \eqref{08} follows from the solvability of \eqref{09}. In particular,  Theorem \ref{05} includes the following earlier result of Wang-Yang-Yau \cite{WYY1} as a special case.

\begin{corollary} \textup{\bf (Wang-Yang-Yau \cite[Theorem 1.1] {WYY1})}\label{10}
 Suppose that $E$ admits a Hermitian metric $\hat{h}$ such that its Hermitian-Yang-Mills tensor $\Lambda_\omega R^{\hat h}$ is positive. Then, for any Hermitian metric $\xi$, there exists a unique Hermitian metric $h$ on $E$ such that $\Lambda_\omega R^h=\xi h^{-1}$.
\end{corollary}

By using the above results, and starting with an arbitrarily fixed Hermitian metric $h_0$, we can inductively define a sequence of Hermitian metrics $\lfb h_m\rfb$ such that
\begin{eqnarray}\label{20}
\Lambda_\omega R^{h_m}=h_{m-1}h_m^{-1}, \qquad m\geq 1.
\end{eqnarray}

Our first main result of this paper is
\begin{theorem}\label{main30}
Suppose that $E$ admits a Hermitian metric $\hat h$ such that its Hermitian-Yang-Mills tensor $\Lambda_\omega R^{\hat h}$ is positive definite.
If $E$ is stable, then the sequence $\lfb h_m\rfb$ given by \eqref{20} converges to the Hermitian-Einstein metric on $E$.
\end{theorem}

\begin{rmk}\label{d10}
Note that taking trace in \eqref{5} and integrating give
\[
\mu =
\frac{\pi\,\mu(E)} {\operatorname{Vol}(X,\omega)}.
\]
Since we assume $E$ admits a Hermitian metric $\hat h$ with $ \Lambda_\omega R^{\hat h}>0$, $\mu$ is necessarily positive. Without loss of generality, we may scale $\omega$ to make $\mu=1$ so that the iteration procedure takes its simple form \eqref{20}.\footnote{In general, the slope stability depends on the \ka class of the base manifold. However, it is easy to see that if $E$ is $\omega$-stable, then $E$ is $t\omega$-stable for $t>0$.} If we do not make this normalization, then we would have to modify the iteration procedure \eqref{20}.
\end{rmk}

\begin{rmk}\label{i40}
If $X$ is projective, then the positivity assumption in Theorem \ref{10} holds automatically. This is because, in this case, we can replace $E$ by $E^m:=E\otimes L^m$ for some ample line bundle $L$ over $X$, and $E^m$ is stable if and only if $E$ is; hence $E^m$ admits \he metric if and only if $E$ does. It is obvious that $E^m$ admits a metric $h$ such that $ \Lambda R^h$ is positive. In fact, we can get a more general result by using Theorem \ref{05}, but we prefer this neat version.
\end{rmk}

\begin{rmk} In 1978, Yau \cite{yau} initiated an iteration process to study the complex Monge-Amp\'ere equation by using the solution to the Calabi conjecture. It is a general scalar version of the (inverse) K\"ahler-Ricci iteration $Ric_{\omega_m}=-\omega_{m-1}$. As a special case, the existence and uniqueness of negative \ke metric was obtained by using this iteration scheme in \cite{yau}.  Decades later, on a Fano manifold $(X,\omega_0)$, the convergence problem of the analogous (inverse) K\"ahler-Ricci iteration sequence of K\"ahler forms \{$\omega_m$\} given by
\[
\text{Ric}_{\omega_m}=\omega_{m-1} \qquad (m\geq 1)
\]
to the Fano K\"ahler-Einstein form was studied; see Nadel \cite{na}, Rubinstein \cite{yr}, Keller \cite{jk},  Berman et al \cite{B et al}, Darvas and Rubinstein \cite{DR} and the references therein, as well as  Zhang \cite{kz} for the cscK case.
\end{rmk}

Furthermore, based on the recent work of Fan-Wang-Yang-Yau \cite{WYY2}, by slightly modifying the proof of Theorem \ref{main30} we generalize it to Higgs bundles over compact K\"ahler manifolds; see Theorem \ref{415}.

The paper is organized as follows. In Section \ref{s20}, we define a modified Donaldson type functional and use it to show that the Donaldson functional from \cite{sd1} is decreasing along the iteration sequence. In Section \ref{s30}, we derive necessary uniform estimates and prove Theorem \ref{main30}. In Section \ref{s50}, we briefly describe the modifications of our proof of Theorem \ref{main30} in the presence of Higgs fields and prove Theorem \ref{415}. Finally, in the Appendix,  we give a new proof of Theorem \ref{05} of Wang-Yang-Yau \cite{WYY3}, as well as its extension to Higgs bundles, by using the heat flow method.

\section{The Donaldson Functional under iteration}\label{s20}

\subsection{Notation}
We first fix notation. Let $E$ be a holomorphic vector bundle of rank $r$ over a compact \ka manifold $(X,\omega)$ of dimension $n$, and let $\mathcal H$ denote the space of Hermitian metrics on $E$.  Let $\ud{e}=\lb e_1,\cdots,e_r\rb$ be a local holomorphic frame of $E$, and $z_1,\cdots,z_n$ be any local holomorphic coordinates on $X$ such that the \ka form is given by $\omega=\frac{\ii}{2}g_{i\bar j}dz_i\wedge d\bar z_j$.
Then, for any $h=h_{\alpha\bar{\beta}}e^\alpha\otimes \bar{e^{\beta}}\in\mathcal H$, the curvature $R^h$ of its Chern connection, viewed as an $\End(E)$-valued $(1,1)$-form, is given by
\[
R^h=\frac{\ii}{2}R_{i\bar j\alpha}^\gamma e_\gamma\otimes e^\alpha\otimes dz_i\wedge d\bar z_j,
\]
where
\[ R_{i\bar j\alpha}^\gamma = h^{\gamma\bar\beta}R_{i\bar j\alpha\bar\beta} \quad \text{and} \quad  R_{i\bar j\alpha\bar\beta}=-\frac{\pa^2 h_{\alpha\bar\beta}}{\pa z_i\pa\bar z_j}+ h^{\gamma\bar\delta}\frac{\pa h_{\alpha\bar\delta}}{\pa z_i}\frac{\pa h_{\gamma\bar\beta}}{\pa\bar z_j}.
\]

If we take trace of $R^h$, then  we get
\[
c_1(E)=\frac{1}{\pi}\lsb \tr\lb R^h\rb\rsb=\lsb \frac{\ii}{2\pi}\sum_\ga R_{i\bar j\ga}^\ga dz_i\wedge d\bar z_j\rsb.
\]
The $\omega$-degree of $E$ is given by
\[
\deg_\omega(E)=\int_X c_1(E)\wedge\frac{\omega^{n-1}}{(n-1)!}
\]
and the slope is
\[
\mu(E)=\frac{\deg_\omega(E)}{\operatorname{rk}(E)} =\frac{\deg_\omega(E)}{r}.\]

We also define the contraction of $R^h$ with the \ka form $\omega$ by
\[
\Lambda_\omega R^h=R_\gm^\ga e_\ga\otimes e^\gm =g^{i\bar j}R_{i\bar j\gm}^\ga e_\ga\otimes e^\gm,
\]
which is a self-adjoint endomorphism of $E$ with respect to $h$. Here, for the matrix $\lsb R_\gm^\ga\rsb$, we view $\gm$ as the row index and $\ga$ as the column index.

\subsection{The Donaldson functional and modified Donaldson functional}
Let $I=[0,1]$. For any two Hermitian metrics $h_0,h_1\in\hh$, let $\lfb h_t\rfb_{t\in I}$ be any smooth path in $\mathcal H$ joining $h_0$ and $h_1$. The Donaldson functional introduced in \cite{sd1} is given by
\begin{align}\label{d20}
\begin{split}
\M\lb h_0,h_1\rb &=\int_0^1 \lb \int_X \tr \lb \dot h_t h_t^{-1}\lb \Lambda_\omega R^{h_t}-\mu I_E\rb\rb dV\rb dt\\
&\overset{\mu=1}{=}\int_0^1 \lb \int_X \tr \lb \dot h_t h_t^{-1}\lb \Lambda_\omega R^{h_t}- I_E\rb\rb dV\rb dt,
\end{split}
\end{align}
where $\dot h_t=\frac{d}{dt}h_t$ and $dV=\frac{\omega^{n}}{n!}$ is the volume form.
A key fact is that the above integral is independent of the chosen path $\lfb h_t\rfb$.

For any Hermitian metric $\xi$ on $E$, we define the modified Donaldson functional as
\begin{eqnarray}\label{d30}
\M_{\xi}\lb h_0,h_1\rb=\int_0^1 \lb \int_X \tr \lb \dot h_t h_t^{-1}\lb \Lambda_\omega R^{h_t}-\xi h_t^{-1}\rb\rb dV\rb dt.
\end{eqnarray}

\begin{lemma}\label{d40}
We have
\[
\M_{\xi}(h_0,h_1)=\M(h_0,h_1)+ \int_X \log\frac{\det h_1}{\det h_0}\, dV+\int_X \lb \tr\lb \xi h_1^{-1}\rb- \tr\lb \xi h_0^{-1}\rb\rb dV.
\]
In particular, the modified functional $\M_{\xi}$ is independent of the chosen path $\lfb h_t\rfb$.
\end{lemma}

\begin{proof}

By definition, we have
\begin{align*}
\begin{split}
\M_{\xi}(h_0,h_1)=& \M(h_0,h_1)+ \int_0^1 \lb \int_X \tr  \lb \dot  h_t h_t^{-1} \rb dV\rb dt\\
& -\int_0^1 \lb \int_X \tr  \lb \dot h_t h_t^{-1} \xi h_t^{-1}\rb dV\rb dt.
\end{split}
\end{align*}
Thus, the lemma follows from the fact that
\begin{align*}
\begin{split}
\int_0^1 \lb \int_X \tr  \lb\dot h_t h_t^{-1} \rb dV\rb dt & =\int_X \lb \int_0^1\frac{d}{dt}\log\det h_t\rb dt dV\\
& =\int_X \log\frac{\det h_1}{\det h_0}\, dV
\end{split}
\end{align*}
and
\begin{align*}
\begin{split}
\int_0^1 \lb \int_X \tr  \lb \dot h_t h_t^{-1}\xi h_t^{-1}\rb dV\rb dt &=  -\int_X \int_0^1\frac{d}{dt}  \tr\lb \xi h_t^{-1}\rb dt dV\\
&=-\int_X \lb \tr\lb \xi h_1^{-1}\rb- \tr\lb \xi h_0^{-1}\rb\rb dV.
\end{split}
\end{align*}
\end{proof}

Next, we define a  $L^2$ metric on $\hh$ as in \cite{ko}. For any $h\in\hh$ and $\phi, \psi\in T_h\hh$,  we let
\[
\langle \phi, \psi\rangle_{L^2}(h)=\int_X \tr\lb \phi h^{-1}\psi h^{-1}\rb dV.
\]
We note that $\lfb h_t\rfb$ is a geodesic if and only if \[\frac{d}{dt}\lb \dot h_t h_t^{-1} \rb=0.\]
The space $\hh$ with this $L^2$ metric is geodesic complete.

\begin{lemma}\label{d50}
Let $\xi\in\hh$ be a Hermitian metric and let $h\in\hh$ be the unique solution of $\Lambda R^h=\xi h^{-1}$. Then $\M_\xi(\xi,h)\leq \M_\xi (\xi, \xi)=0$.
\end{lemma}

\begin{proof}

Let $\lfb h_t\rfb_{t\in I}$ be the unique geodesic in $\hh$ such that $h_0=\xi$ and $h_1=h$. Set \[\rho=\dot h_t h_t^{-1} \qquad \rm{and} \qquad \eta(t)=\M_\xi(\xi,h_t).\] We note that $\rho$ is constant in $t$ due to the geodesic equation $\frac{d}{dt}\rho=0$. Then, by a direct computation, we have
\[
\eta'(t)=\int_X \tr \lb \rho \lb \Lambda_\omega R^{h_t}-\xi h_t^{-1}\rb \rb dV.
\]
In particular, $\eta'(1)=0$ so $h$ is a critical point of $\eta(t)$. Now, by direct computations and integration by parts, we obtain
\begin{equation}\label{55}
\eta''(t)=\int_X \lb |\bar\pa\rho|^2+\tr\lb \rho^2 \xi h_t^{-1}\rb\rb dV\geq 0,
\end{equation}
where we have used the fact that $\rho$ is self-adjoint and $\xi$ is positive definite. This implies that $\eta$ is convex and has a critical point at $1$ so $\eta(1)\leq \eta(0)$.

\end{proof}

\begin{lemma}\label{d60} Suppose the sequence of Hermitian metrics $\{h_m\}\subset \hh$ is given by \eqref{20}. Then,
for any fixed $\xi\in\hh$, we have $\M(\xi, h_m)\leq \M(\xi, h_{m-1})$.
\end{lemma}

\begin{proof}

Set $\M(h) :=\M(\xi,h)$. Then we have
\begin{align}\label{d70}
\begin{split}
\M(h_m) =  & \M(\xi,h_{m-1})+\M(h_{m-1},h_m)\\
= & \M(h_{m-1})+\M_{h_{m-1}}\lb h_{m-1},h_m\rb\\
& - \int_X\lb  \log \frac{\det h_m}{\det h_{m-1}}+\tr\lb h_{m-1}h_m^{-1}\rb -r\rb dV.
\end{split}
\end{align}

At each point $x_0\in X$, we choose an orthonormal frame of $E$ with respect to $h_{m-1}$ and let the eigenvalues of $h_m$ be $\lambda_1,\cdots,\lambda_r$ (which are all positive). Then, at $x_0$, we have
\[
\log \frac{\det h_m}{\det h_{m-1}}+\tr\lb h_{m-1}h_m^{-1}\rb -r=\sum_{j=1}^r \lb \log \lambda_j+\frac{1}{\lambda_j}-1\rb.
\]
By a simple computation, one finds that the minimum of the function
\[g(x)=\log x+\frac 1x-1 \qquad (x>0)\] is achieved at $x=1$ with $g(1)=0$. Thus, we have
\[
\int_X\lb  \log \frac{\det h_m}{\det h_{m-1}}+\tr\lb h_{m-1}h_m^{-1}\rb -r\rb dV\geq 0.
\]
Meanwhile,  by Lemma \ref{d50}, we also have $\M_{h_{m-1}}\lb h_{m-1},h_m\rb\leq 0$. Therefore, it follows that $\M(h_m)\leq \M(h_{m-1})$.

\end{proof}

\section{Convergence of the sequence $\{h_m\}$}\label{s30}

First of all, we  fix any $h_0\in\hh$ and let $\lfb h_m\rfb$ be the sequence given by inductively solving the equation \eqref{20}. In this section, we show that the sequence $\lfb h_m\rfb$ converges to the \he metric, thus proving Theorem \ref{main30}.  Throughout the section, we denote $\M(\cdot)=\M\lb h_0,\cdot\rb$ by the Donaldson functional.

\subsection{The $C^0$ estimate} 

\begin{theorem}\label{c1000}
There exists a constant $C>0$, independent of $m$, such that
\[
\frac 1C h_0\leq h_m\leq C h_0
\]
for all $m\geq 1$.
\end{theorem}

\begin{proof}

Set $K_m=\Lambda_\omega R^{h_m}$, which is positive and self-adjoint with respect to both $h_{m-1}$ and $h_{m}$. By using equation \eqref{20}, a direct computation shows that
\[
\Lambda_\omega R^{h_{m-1}}=\Lambda_\omega R^{h_{m}}+\Lambda_\omega\bar\pa \lb \nabla^{h_m}K_m K_m^{-1}\rb,
\]
or equivalently
\begin{eqnarray}\label{200}
K_{m-1}=K_m-g^{i\bar j}\pa_{\bar j}\lb \nabla_i^{h_m}K_m K_m^{-1} \rb.
\end{eqnarray}

For each $m\geq 1$, we let $\bar\lambda_x\lb K_m\rb$ be the largest eigenvalue of $K_m$ at $x\in X$ with respect to $h_m$ and define \[\bar\lambda\lb K_m\rb :=\sup_{x\in X}\bar\lambda_x\lb K_m\rb.\]
Pick $x_m\in X$ such that $\bar\lambda\lb K_m\rb= \bar\lambda_{x_m}\lb K_m\rb$, and let $v\in E_{x_m}$ be an eigenvector of $K_m$ at $x_m$ with respect to $ \bar\lambda_{x_m}\lb K_m\rb$ with $\Vert v\Vert_{h_m}=1$. Evaluating at $x_m$, by equation \eqref{200}, we have
\begin{align}\label{300}
\begin{split}
\langle	K_{m-1}v,v\rangle_{h_m}=& \langle	K_{m}v,v\rangle_{h_m}- \langle g^{i\bar j}\pa_{\bar j}\lb \nabla_i^{h_m}K_m K_m^{-1} \rb v, v\rangle_{h_m} \\
=& \bar\lambda\lb K_m\rb- \langle g^{i\bar j}\pa_{\bar j}\lb \nabla_i^{h_m}K_m K_m^{-1} \rb v, v\rangle_{h_m}.
\end{split}
\end{align}

Now, we consider the term $ \langle g^{i\bar j}\pa_{\bar j}\lb \nabla_i^{h_m}K_m K_m^{-1} \rb v, v\rangle_{h_m}$. First, we extend $v$ to a smooth local section $\wt v$ of $E$ over a neighborhood $U$ of $x_m$ such that $\nabla^{h_m}\wt v(x_m)=0$ and $\bar\pa \wt v(x_m)=0$. Then, we set  \[f(x)=\langle K_m \wt v,\wt v\rangle_{h_m}-\bar\lambda\lb K_m\rb \langle \wt v,\wt v\rangle_{h_m} \in C^\infty \lb U,\R\rb.\] By definition of $f$, we see that $f(x)\leq 0$ and $f(x_m)=0$. Thus, $\pa_i f(x_m)=0$, $\pa_{\bar j} f(x_m)=0$ and $\Delta  f(x_m)=g^{i\bar j}\pa_i\pa_{\bar j}f(x_m)\leq 0$. By using the fact that $K_m$ is $h_m$-self adjoint and $K_m v=\bar\lambda\lb K_m\rb v$, we get
\[
\langle g^{i\bar j}\pa_{\bar j}\lb \nabla_i^{h_m}K_m K_m^{-1} \rb v, v\rangle_{h_m}\leq 0.
\]
It follows that
\[
\langle	K_{m-1}v,v\rangle_{h_m}\geq  \langle	K_{m}v,v\rangle_{h_m}=\bar\lambda\lb K_m\rb.
\]
Now, we have
\begin{align*}
	\begin{split}
\langle K_{m-1}v,v\rangle_{h_{m-1}} &=\langle h_{m-1}h_m^{-1} K_{m-1}v,v\rangle_{h_{m}}\\
& =\langle K_m K_{m-1}v,v\rangle_{h_{m}}  \\
&=\langle  K_{m-1}v,K_m v\rangle_{h_{m}} \\
& =\bar\lambda\lb K_m\rb \langle K_{m-1}v,v\rangle_{h_{m}}
\geq \bar\lambda\lb K_m\rb^2
\end{split}
\end{align*}
and
\[
\langle  v,v\rangle_{h_{m-1}}=\langle K_{m}v,v\rangle_{h_{m}}=\bar\lambda\lb K_m\rb.
\]
Hence,
\[
\bar\lambda\lb K_{m-1}\rb\geq \frac{\langle K_{m-1}v,v\rangle_{h_{m-1}}}{\langle  v,v\rangle_{h_{m-1}}}\geq \bar\lambda\lb K_m\rb.
\]
Therefore, $\bar\lambda\lb K_m\rb\leq \bar\lambda\lb K_0\rb$ for all $m\geq 1$.
Moreover, since $\Lambda_\omega R^{h_m}$ is $h_m$-self adjoint,
\[
\left |\Lambda_\omega R^{h_m}\right |_{h_m}=\lb \tr \lb \Lambda_\omega R^{h_m}\rb^2 \rb^{\frac 12}\leq \sqrt r \bar\lambda\lb K_m\rb.
\]
 Hence, there exists a constant $C_1>0$ such that, for all $m\geq 1$,
\[
\left |\Lambda_\omega R^{h_m}\right |_{h_m}\leq C_1.
\]

Next, let
\[
\phi_m=\log\frac{\det h_m}{\det h_0}=\tr\log \lb h_mh_0^{-1}\rb, \quad \ga_m=\frac 1V \int \phi_m\ dV, \quad \rm{and} \quad \wt h_m=e^{-\frac{\phi_m}{r}}h_m.
\]
Then, we have $\det \wt h_m=\det h_0$ and that
\begin{eqnarray}\label{400}
\Delta \phi_m=\tr \lb\Lambda_\omega R^{h_0}\rb-\tr \lb\Lambda_\omega R^{h_m}\rb.
\end{eqnarray}
Since $\tr \lb \Lambda_\omega R^{h_m}\rb=\tr \lb h_{m-1}h_m^{-1}\rb>0$, we get
\begin{eqnarray}\label{500}
\tr \lb \Lambda_\omega R^{h_0}\rb -C_1\sqrt r\leq\Delta\phi_m\leq \tr \lb\Lambda_\omega R^{h_0}\rb.
\end{eqnarray}
On the other hand, a direct computation gives
\[
\Lambda_\omega R^{\wt h_m}=\Lambda_\omega R^{h_m}+\frac{\Delta\phi_m}{r}I,
\]
which implies that $\bar\lambda \lb \Lambda_\omega R^{\wt h_m}\rb\leq C_2.$
Thus, there exists $C_3>0$ such that
\[\left |\Lambda_\omega R^{\wt h_m}\right |_{\wt h_m}\leq C_3\]
for all $m\geq 1$.
We note that $\log \wt h_mh_0^{-1}$ is trace-free.

By \eqref{400}, we also have
\[
\Delta \phi_m=\tr \lb \Lambda_\omega R^{h_0}\rb-\tr\lb h_{m-1}h_m^{-1}\rb.
\]
It follows that
\[
\int_X \tr\lb h_{m-1}h_m^{-1}\rb dV=\int_X \tr\lb\Lambda_\omega R^{h_0}\rb dV=rV,
\]
where recall that we have scaled $\omega$ such that $\mu=1$. Now,
\begin{align*}
	\begin{split}
\tr\lb h_{m-1}h_m^{-1}\rb & \geq r \det \lb h_{m-1}h_m^{-1}\rb^{\frac 1r}\\
& =re^{\frac{\phi_{m-1}-\phi_m}{r}}.
\end{split}
\end{align*}
By Jensen's inequality, we have
\begin{align*}
	\begin{split}
rV  =\int_X \tr\lb h_{m-1}h_m^{-1}\rb dV & \geq r \int_X e^{\frac{\phi_{m-1}-\phi_m}{r}} dV \\
& \geq rV e^{\frac 1V\int_X \frac{\phi_{m-1}-\phi_m}{r}dV}\\
& =rV e^{\frac{1}r (\ga_{m-1}-\ga_m)}.
\end{split}
\end{align*}
Thus $\ga_{m-1}\leq\ga_m$. Since $\phi_0=0$, we know that $\ga_0=0$ and thus $\ga_m\geq 0$.

By \eqref{d70}, we have
\[
\M\lb h_{m-1}\rb-\M\lb h_{m}\rb=\int_X \lb \log\frac{\det h_m}{\det h_{m-1}}+\tr\lb h_{m-1}h_m^{-1}\rb -r\rb dV-\M_{h_{m-1}}\lb h_{m-1}, h_m\rb.
\]
Since $\M_{h_{m-1}}\lb h_{m-1}, h_m\rb\leq 0$ and $\int_X \tr\lb h_{m-1}h_m^{-1}\rb dV=rV$, we get
\begin{align*}
	\begin{split}
\M\lb h_{m-1}\rb-\M\lb h_{m}\rb & \geq \int_X  \log\frac{\det h_m}{\det h_{m-1}}\\
& =\int_X \lb \phi_m-\phi_{m-1}\rb dV \\
& =\lb \ga_m-\ga_{m-1}\rb V.
\end{split}
\end{align*}
In \cite{sd1} and \cite{sd2}, Donaldson showed that $\M (\cdot) $ is bounded from below if $E$ is semi-stable. Combining all of these, we conclude that
\[
\ga_m V\leq \M\lb h_{0}\rb-\M\lb h_{m}\rb \leq C_4
\]
for some constant $C_4>0$ and all $m\geq 1$. Therefore, for all $m\geq 1$,
\[0\leq \ga_m\leq \frac{C_4}{V}.\]

Furthermore, a simple computation shows that
\begin{align*}
\begin{split}
\M\lb \wt h_m,h_m\rb & =\int_X \lb\frac{\phi_m}{r}\lb \tr(\Lambda_\omega R^{\wt h_m})-r\rb-\frac{1}{2r} \phi_m\Delta\phi_m\rb dV\\
& =\int_X \lb\frac{\phi_m}{r}\lb \tr(\Lambda_\omega R^{ h_0})-r\rb-\frac{1}{2r} \phi_m\Delta\phi_m\rb dV.
\end{split}
\end{align*}

Let $\psi=\frac 1 r\tr \lb\Lambda_\omega R^{ h_0}\rb-1$. Then, we have $\int_X \psi dV=0$. It follows that
\begin{align}\label{550}
\begin{split}
\M(h_m)& =\M\lb \wt h_m\rb+\M\lb \wt h_m,h_m\rb \\
 &=\M\lb \wt h_m\rb+ \int_X \phi_m\psi\, dV +\frac{1}{2r}\Vert\bar\pa\phi_m\Vert_{L^2}^2.
\end{split}
\end{align}
Since $\M\lb h_m\rb\leq \M\lb h_0\rb=0$ by Lemma \ref{d60} and $\M\lb \wt h_m\rb\geq -C_5$, we get
\[
\frac{1}{2r}\Vert\bar\pa\phi_m\Vert_{L^2}^2 +\int_X \phi_m\psi\, dV \leq C_5.
\]
By the Poincar\`e inequality, we obtain
\begin{align*}
\begin{split}
\left | \int_X \phi_m \psi\, dV\right | & =\left | \int_X \lb \phi_m-\ga_m\rb \psi\, dV\right |\\
& \leq \Vert \phi_m-\ga_m\Vert_{L^2}\Vert\psi\Vert_{L^2} \\
& \leq C_6\Vert \bar\pa\phi_m\Vert_{L^2}
\end{split}
\end{align*}
for some constant $C_6>0$ and all $m\geq 1$.

It follows that
\[
\frac{1}{2r}\Vert\bar\pa\phi_m\Vert_{L^2}^2-C_6\Vert \bar\pa\phi_m\Vert_{L^2}\leq C_5
\]
and hence
\[
\Vert \bar\pa\phi_m\Vert_{L^2}\leq C_7
\] for some constant $C_7>0$ and all $m\geq 1$.

Since we have uniform control on $\ga_m=\frac 1V \int \phi_m\ dV$ for all $m\geq 1$, we get
\[\Vert \phi_m\Vert_{L^2}\leq C_8\]
as well, and thus there exists a constant $C_9>0$ such that, for all $m\geq 1$,
\begin{eqnarray}\label{600}
\Vert\phi_m\Vert_{W^{1,2}}\leq C_9.
\end{eqnarray}
Combining \eqref{500} and \eqref{600}, the standard Green function estimates and the Moser iteration guarantee that there exists $C_{10}>0$ such that
\begin{eqnarray}\label{700}
\Vert\phi_m\Vert_{C^0}\leq C_{10}.
\end{eqnarray}

By \eqref{550}, we have
\begin{align*}
\begin{split}
\M(\wt h_m) & =\M\lb h_m\rb-\frac{1}{2r}\Vert\bar\pa\phi_m\Vert_{L^2}^2 -\int_X \phi_m\psi dV\\
& \leq \Vert \phi_m\Vert_{L^2} \Vert \psi\Vert_{L^2}\\
& \leq C_9\Vert \psi\Vert_{L^2}.
\end{split}
\end{align*}
By Donaldson's work \cite{sd1, sd2}, we know that the stability of $E$ implies the functional $\M$ is proper in the sense that the $L^1$ norm of $\log\lb \wt h_mh_0^{-1}\rb$ is controlled by $\M\lb\wt h_m\rb$. This was packaged into the $C^0$ estimate in \cite{sim1988} by using Uhlenbeck-Yau's work \cite{UY}. By \cite[Proposition 5.3]{sim1988}, we have
\[
\sup_X \left |\log\lb \wt h_mh_0^{-1}\rb\right | \leq C_{11}+C_{12}\M\lb \wt h_m\rb\leq C_{13}
\]
for some constant $C_{13}>0$ and all $m\geq 1$.
This implies that, $\wt h_m$ and $h_0$ are uniformly equivalent. Finally, it follows from \eqref{700} that $h_m$ and $\wt h_m$ are uniformly equivalent. This completes the proof of   Theorem \ref{c1000}.

\end{proof}

\subsection{The higher order estimates}

We now derive uniform $C^{k,\ga}$ estimates of $h_m$ for $k\geq 1$ and $0<\ga<1$.  Denote by  $H_m=h_mh_0^{-1}$,  we have
\begin{align}\label{800}
\begin{split}
\Lambda_\omega \bar\pa \lb \nabla^{h_0}H_m H_m^{-1}\rb=\Lambda_\omega R^{h_m}-\Lambda_\omega R^{h_0}=h_{m-1}h_m^{-1}-\Lambda_\omega R^{h_0}.
\end{split}
\end{align}

\begin{theorem}\label{c2000}
For any $k\geq 1$ and $0<\ga<1$,  there exists a constant $C_{k,\ga}>0$ such that
\[
\Vert h_m\Vert_{C^{k,\ga}\lb h_0\rb}\leq C_{k,\ga}.
\]
\end{theorem}

\begin{proof}

It suffices to show that $\lfb \Vert H_m\Vert_{C^{k,\ga}}\rfb$ is uniformly bounded.
By \eqref{800} and direct computations, we have
\begin{align}\label{900}
\begin{split}
\Delta_{h_0} H_m=g^{i\bar j}\nabla_i H_m H_m^{-1}\pa_{\bar j} H_m+\lb h_{m-1}h_m^{-1}-\Lambda_\omega R^{h_0}\rb H_m.
\end{split}
\end{align}
Note that the $C^0$ estimate in Theorem \ref{c1000} gives a uniform bound on $H_m$ and $H_m^{-1}$. It then follows that $\lb h_{m-1}h_m^{-1}-\Lambda_\omega R^{h_0}\rb H_m$ is uniformly bounded and so is its $L^p$ norm. By using the interpolation inequality to handle the term $g^{i\bar j}\nabla_i H_m H_m^{-1}\pa_{\bar j} H_m$, and the fact that $H_m^{-1}$ is uniformly bounded, the elliptic regularity theory implies that, for any $p>0$, there exists constant $C_p>0$ such that
\[
\Vert H_m\Vert_{W^{2,p}}\leq C_p.
\]
It follows from the Sobolev embedding that, for any $0<\alpha<1$,  there exists  $C_{\alpha}>0$ such that
\[\Vert H_m\Vert_{C^{1,\ga}\lb h_0\rb}\leq C_{1,\ga}.\]

Now, we can prove the rest of higher order estimates by induction. Assuming we have obtained the uniform $C^{k,\ga}$ estimates on $\lfb H_m\rfb$. Then, we know that the right hand side of \eqref{900} is uniformly bounded in $C^{k-1,\ga}$ norm. Therefore, it follows from the standard Schauder theory that $\lfb H_m\rfb$ is uniformly bounded in $C^{k+1,\ga}$ norm.

\end{proof}

\subsection{The convergence} Now, we are ready to prove the convergence of $\lfb h_m\rfb$.

\begin{theorem}\label{c3000}
The sequence $\lfb h_m\rfb$ converges smoothly to a limiting \he metric $h_\infty$ with $\Lambda_\omega R^{h_\infty}=I_E$.
\end{theorem}

\begin{proof}
By \eqref{d70}, we have
\begin{eqnarray}\label{910}
\M\lb h_{m-1}\rb-\M\lb h_{m}\rb \geq \int_X\lb  -\log\det K_m +\tr K_m -r\rb dV\geq 0,
\end{eqnarray}
where $K_m=h_{m-1}h_m^{-1}$ as before. Hence, for $m\geq 1$, we see that
\[
\M\lb h_{0}\rb-\M\lb h_{m}\rb \geq \sum_{j=1}^m \int_X\lb  -\log\det K_j +\tr K_j -r\rb dV.
\]
Since $\M(\cdot)$ is bounded from below, we know that the series
\[\sum_{j=1}^\infty \int_X\lb  -\log\det K_j +\tr K_j -r\rb dV\] is convergent.
Thus, we have
\[
\lim_{m\to\infty}\int_X\lb  -\log\det K_m +\tr K_m -r\rb dV=0.
\]
Note that $-\log\det K_m +\tr K_m -r \geq 0$ pointwise, and it is uniformly Lipschitz due to the uniform $C^0$ estimate, i.e., $C^{-1}h_0\leq h_m\leq C h_0$, and the uniform $C^{k,\ga}$ estimates.
This forces
\[
\lim_{m\to\infty}h_{m-1}h_m^{-1}=I_E
\]
in $C^0$ topology. Now, a standard fact in PDE guarantees that if a sequence converges in $C^0$ norm and all $C^{k,\ga}$ norms of the sequence are bounded by some constant $M_{k,\ga}$, then the interpolation inequality implies that the sequence converges smoothly.  Therefore, it follows from Theorem \ref{c1000} and Theorem \ref{c2000} that $h_{m-1}h_m^{-1}\to I_E$ smoothly.

On the other hand, by the Arzela-Ascoli theorem, we can find a subsequence $\lfb h_{m_j}\rfb$ of $\lfb h_m\rfb$ which converges smoothly to a limiting Hermitian metric $h_\infty$ on $E$. Note that we have a uniform bound on $\lfb H_m\rfb$ which forces $h_\infty$ to be nondegenerate. By the uniform $C^{k,\ga}$ bounds of $\lfb h_m\rfb$ and by using interpolation inequalities,  we know that $\{h_{m_j}\} \to h_\infty$ smoothly. Thus,
\[
\Lambda_\omega R^{h_\infty}=\lim_{j\to\infty}\Lambda R^{h_{m_j}}=\lim_{j\to\infty}h_{m_{j-1}}h_{m_j}^{-1}=I_E.
\]
Namely, $h_\infty$ is a \he metric.

It remains to show that such limit metric is unique. Since the sequence $\lfb \ga_m\rfb$ is nondecreasing and $0\leq \ga_m\leq \frac{C_4}{V}$, we know that it is convergent and let $\ga_{\infty}$ be its limit. Then, we have
\[
\ga_{\infty}=\lim_{j\to\infty}\ga_{m_j}=\lim_{j\to\infty}\frac 1V \int_X \log \frac{\det h_{m_j}}{\det h_0}dV=\frac 1V \int_X \log \frac{\det h_{\infty}}{\det h_0}dV.
\]
Suppose that there is another subsequence $\lfb h_{n_s}\rfb$ of $\lfb h_m\rfb$ which converges smoothly to another limiting metric $h_\infty'$. Then, we know that $h_\infty'$ is also \he and
\[
\frac 1V \int_X \log \frac{\det h_{\infty}}{\det h_0}dV=\ga_{\infty}.
\]
Since we know that $h_\infty'=ch_\infty$ for some constant $c>0$, the above argument forces $c=1$. Thus, the limit metric is unique.

Since all convergent subsequences of $\lfb h_m\rfb$ converge to $h_\infty$, and we have uniform $C^{k,\ga}$ estimates on $\lfb h_m\rfb$, we can conclude that $h_m\to h_\infty$ smoothly. This completes the proof of Theorem \ref{c3000} and Theorem \ref{main30}.

\end{proof}


\section{The extension to Higgs Bundles}\label{s50}

In this section, based on the recent work \cite{WYY2}, we generalize our Theorem  \ref{main30} to Higgs bundles over compact K\"ahler manifolds.

Recall that a Higgs bundle over a compact K\"ahler manifold $(X,\omega)$ consists of a pair $(E, \Phi)$, where $E$ is a holomorphic vector bundle over $(X,\omega)$, and $\Phi$ is a {\it Higgs field} on $E$, i.e., $\Phi \in H^0(X, \Omega_X^1 \otimes \mathrm{End}(E))$ is a holomorphic 1-form with values in endomorphisms of $E$ satisfying the integrability condition $\Phi \wedge \Phi = 0.$

Given any Hermitian metric $h$ on $E$, we denote by $R^h$ the curvature of its Chern connection and define the {degree}  $\deg(E)$ of $E$ and the slope $\mu(E) $ of $E$ as before.
For the Higgs bundle $(E,\Phi)$, we have the following basic notions:

\begin{itemize}

\item[{$\bullet$}]
$(E,\Phi)$ is {\it stable} if  $\mu(F) < \mu(E)$ for all proper {\it Higgs-invariant} subbundles $F\subset E$. Here, by a subbundle $F\subset E$ being {Higgs-invariant} (or $\Phi$-invariant) it means that $\Phi(F)\subset F \otimes \Omega_X^1.$

\smallskip
\item[{$\bullet$}] A Hermitian metric $h$ on $E$ is said to be a {\it Hermitian-Einstein-Higgs} metric if it satisfies the following equation:
\[\Lambda_\omega (R^h + [\Phi,\Phi^{*,h}]) = \mu I_E,\]
where $\Phi^{*,h}$ is the adjoint of $\Phi$ with respect to the metric $h$. 

\end {itemize}

In \cite{sim1988}, Simpson generalized the Donaldson-Uhlenbeck-Yau theorem to Higgs bundles: {\it a Higgs bundle is (poly)stable if and only if it admits an Hermitian-Einstein-Higgs metric}.

\smallskip
More recently, in \cite{WYY2}, the authors also generalized Corollary \ref{10} (Wang-Yang-Yau \cite[Theorem 1.1] {WYY1}) to Higgs bundles.

\begin{theorem} \textup{\bf (Fan-Wang-Yang-Yau \cite{WYY2})}\label{410}
 Suppose that $(E, \Phi)$ admits a Hermitian metric $\hat h$ such that its Higgs Hermitian-Yang-Mills tensor $ \Lambda_\omega \lb R^{\hat h}+ [\Phi,\Phi^{*,\hat h}]\rb$ is positive definite. Then, for any Hermitian metric $\xi$ on $E$, there exists a unique Hermitian metric $h$ on $E$ such that
$ \Lambda_\omega \lb R^{h}+ [\Phi,\Phi^{*,h}]\rb=\xi h^{-1}$.
\end{theorem}

Now, in the rest of this section, we consider a Higgs bundle $(E, \Phi)$ over a compact $(X, \omega)$ and assume either $(X, \omega)$ is projective or there exists a Hermitian metric $\hat h$ on $E$ such that $\Lambda_\omega \lb R^{\hat h}+\lsb \Phi,\Phi^{*,\hat h}\rsb\rb>0$.
By using Theorem \ref{410} and induction, starting with an arbitrarily fixed Hermitian metric $h_0$, we can construct a sequence of Hermitian metrics $\lfb h_m\rfb$ by solving
\begin{eqnarray}\label{420}
\Lambda_\omega \lb R^{h_m}+ [\Phi,\Phi^{*,h_m}]\rb=h_{m-1}h_m^{-1}, \qquad m\geq 1.
\end{eqnarray}

By slightly modifying the proof of Theorem \ref{main30}, and normalizing $\mu=1$ as before, we can show that if $(E, \Phi)$ is stable then the iteration procedure \eqref{420} leads to a unique Hermitian-Einstein-Higgs metric.

\begin{theorem} \label{415}
Let $(E, \Phi)$ be a stable Higgs bundle over $(X, \omega)$. Suppose that either $(X, \omega)$ is projective or there exists a Hermitian metric $\hat h$ on $E$ such that $\Lambda_\omega \lb R^{\hat h}+\lsb \Phi,\Phi^{*,\hat h}\rsb\rb>0$. Then, the sequence $\lfb h_m\rfb$ given by \eqref{420} converges smoothly to a Hermitian-Einstein-Higgs metric $h_\infty$ satisfying
\[
\Lambda_\omega \lb R^{h_\infty}+\lsb \Phi,\Phi^{*,{h_\infty}}\rsb\rb=I_E.
\]
\end{theorem}

\begin{proof}  [Sketch of Proof]
Let  $\mathcal H$ denote the space of Hermitian metrics on $E$ as before. We first note that, for any  Hermitian metrics $h_0, h_1\in \mathcal H$ and $\xi \in \mathcal H$, if we let
\[
\M^H\lb h_0,h_1\rb =\int_0^1 \lb \int_X \tr \lb \dot h_t h_t^{-1}\lb \Lambda_\omega \lb R^{h_t}+\lsb\Phi,\Phi^{*,{h_t}}\rsb\rb- I_E\rb\rb dV\rb dt
\]
and
\[
\M_{\xi}^H\lb h_0,h_1\rb=\int_0^1 \lb \int_X \tr \lb \dot h_t h_t^{-1}\lb \Lambda_\omega \lb R^{h_t}+\lsb\Phi,\Phi^{*,{h_t}}\rsb\rb-\xi h_t^{-1}\rb\rb dV\rb dt,
\]
where $\lfb h_t\rfb$ is any smooth path joining $h_0$ and $h_1$, then $\M_{\xi}^H$ and $\M^H$ are independent of the choice of $\lfb h_t\rfb$. Furthermore, if we fix $\xi$ and let $\lfb h_t\rfb$ be the geodesic joining $\xi$ and another Hermitian metric $h$, then direct computations give the following analogue of \eqref{55},
\[
\frac{d^2}{dt^2}\M_{\xi}^H\lb \xi, h_t\rb=\int_X \lb \left |\bar\pa\lb \dot h_t h_t^{-1}\rb\right |^2+\left | \lsb \Phi,\dot h_t h_t^{-1}\rsb \right |_{h_t}^2
+\tr\lb \lb\dot h_t h_t^{-1}\rb^2 \xi h_t^{-1}\rb\rb dV\geq 0,
\]
from which the corresponding version of  Lemma \ref{d50}, hence also of Lemma \ref{d60}, hold in the current setting.  Now, we let \[K_m=\Lambda_\omega F^{h_m}=\Lambda_\omega \lb R^{h_m}+\lsb\Phi,\Phi^{*,{h_m}}\rsb\rb.\] We notice that, similar to formula \eqref{200}, we have
\begin{eqnarray*}
K_{m-1}=K_m-g^{i\bar j}\pa_{\bar j}\lb \nabla_i^{h_m}K_m K_m^{-1} \rb
+\Lambda_\omega \lb \lsb\Phi,\Phi^{*,h_{m-1}}\rsb-\lsb\Phi,\Phi^{*,{h_m}}\rsb\rb.
\end{eqnarray*}
Since $K_m=h_{m-1}h_m^{-1}$, we know that $\Phi^{*,h_{m-1}}=K_m\Phi^{*,h_m}K_m^{-1}$. Assume that at $x_m\in X$ we have $\bar\lambda_{x_m}\lb K_m\rb=\bar\lambda \lb K_m\rb$ and $v\in E_{x_m}$ is a $h_m$-unit eigenvector of $\bar\lambda\lb K_m\rb$, a simple linear algebra computation shows that
\[
\left\langle  \Lambda_\omega\lb \lsb\Phi,\Phi^{*,h_{m-1}}\rsb-\lsb\Phi,\Phi^{*,{h_m}}\rsb\rb_{x_0}(v),v\right \rangle_{h_m}\geq 0.
\]
It follows that \[\bar\lambda \lb K_m\rb\leq \bar\lambda \lb K_{m-1}\rb\] as before, and hence $\left |\Lambda_\omega F^{h_m}\right |_{h_m}\leq C$ is uniformly bounded.

Also note that, since $\tr \lsb\Phi,\Phi^{*,{h_m}}\rsb=0,$ we have
\[\tr\lb \Lambda_\omega R^{h_m}\rb=\tr\lb \Lambda F^{h_m}\rb.\]
Let $\phi_m=\log \frac{\det h_m}{\det h_0}$ as before. Then, we have
\[\Delta \phi_m=\tr\lb \Lambda_\omega F^{h_0}\rb -\tr\lb \Lambda_\omega F^{h_m}\rb =\tr\lb \Lambda_\omega R^{h_0}\rb -\tr\lb \Lambda_\omega R^{h_m}\rb \]
and \[\tr\lb\Lambda_\omega R^{h_m}\rb =\tr\lb \Lambda_\omega F^{h_m}\rb =\tr \lb h_{m-1}h_m^{-1}\rb>0.\]
Therefore, the same argument as in the proof of Theorem \ref{main30} yields the uniform $C^0$ estimate of $\lfb h_m\rfb$. It follows that $\lfb\left | \lsb\Phi,\Phi^{*,{h_m}}\rsb\right |_{h_m}\rfb$ is uniformly bounded.

The higher order estimate and the convergence of $\lfb h_m\rfb$ follow from exactly the same proof as that of Theorem \ref{main30}.
\end{proof}


\section{Appendix: The Heat Flow Approach} \label{s40}

\subsection{The heat flow proof of Theorem \ref{05}}
In this subsection, we give a new proof of Theorem \ref{05} by using the heat flow method. We fix $\lambda<\lambda_{\min}$, an arbitrary initial Hermitian metric $h_0$ on $E$, and another Hermitian metric $\xi$ on $E$. We seek a Hermitian metric $h$, which is the solution to the equation
\begin{equation}\label{a50}
\Lambda_\omega R^h=\lambda I_E+\xi h^{-1},
\end{equation}
by considering the following heat flow: 
\begin{eqnarray}\label{a100}
\begin{cases}
\frac{\pa h_t}{\pa t}h_t^{-1}=-\lb \Lambda_\omega R^{h_t}-\lambda I_E-\xi h_t^{-1}\rb,\\
h(0)=h_0.
\end{cases}
\end{eqnarray}

\begin{theorem}\label{a105}
The heat flow \eqref{a100} exists for all time $0\leq t <\infty$ and, as $t\to\infty$, converges to a unique solution of equation \eqref{a50} smoothly.
\end{theorem}

\begin{proof}
Note that equation \eqref{a100} is equivalent to $\dot h_t=-\Lambda_\omega R^{h_t}h_t+\lambda h_t+\xi$. If we linearize this equation, then we get a strictly parabolic equation, hence the short time existence follows.

Now we need to slightly perturb the modified Donaldson functional to adapt to the current situation. We let
\[
\M_{\xi,\lambda}\lb h_0,h_1\rb=\int_0^1 \lb \int_X \tr \lb \dot h_t h_t^{-1}\lb \Lambda_\omega R^{h_t}-\lambda I_E-\xi h_t^{-1}\rb\rb dV\rb dt
\]
and set $K_t=\Lambda_\omega R^{h_t}-\lambda I_E-\xi h_t^{-1}$. Then, it follows from a direct computation that
\begin{align*}
\begin{split}
\frac{d}{dt}\M_{\xi,\lambda}\lb h_0,h_t\rb & =-\int_X \tr\lb\frac{\pa h_t}{\pa t}h_t^{-1} \lb \Lambda_\omega R^{h_t}-\lambda I_E-\xi h_t^{-1}\rb\rb dV \\
& =-\int_X \Vert K_t\Vert_{h_t}^2 dV\leq 0.
\end{split}
\end{align*}
Namely, the new modified Donaldson functional is decreasing along the flow.

\medskip
\noindent $\bullet$  \textbf{The $C^0$ estimate.}

\smallskip

We fix a Hermitian metric $\hat h$ on $E$ such that $\Lambda_\omega R^{\hat h}-\lambda I_E>0$ and let $H_t=h_t\hat h^{-1}$. Then, we have
\begin{eqnarray}\label{a140}
\Delta \lb \tr H_t\rb \geq \tr \lb H_t \Lambda_\omega R^{\hat h}\rb-\tr \lb H_t \Lambda_\omega R^{h_t}\rb
\end{eqnarray}
and
\[
\frac{d}{dt} \lb \tr H_t\rb =-\tr \lb \lb\Lambda_\omega R^{h_t}-\lambda I_E\rb H_t \rb+\tr\lb \xi \hat h^{-1}\rb.
\]
It follows that
\begin{eqnarray}\label{a150}
\lb \frac{d}{dt}-\Delta\rb\tr H_t\leq \tr\lb \xi  \hat h^{-1}\rb-\tr \lb H_t \lb\Lambda_\omega R^{\hat h}-\lambda I_E\rb\rb.
\end{eqnarray}
Note that both $H_t$ and $\Lambda_\omega R^{\hat h}-\lambda I_E$ are positive. We pick
\[C_1=\sup_X \tr\lb \xi  \hat h^{-1}\rb>0 \quad \rm{and} \quad C_2=\inf_X \lambda_{min}\lb \Lambda_\omega R^{\hat h}-\lambda I_E\rb>0,\]
where $\lambda_{min}\lb \Lambda_\omega R^{\hat h}-\lambda I_E\rb$ is the smallest eigenvalue of $\Lambda_\omega R^{\hat h}-\lambda I_E$. Then,  \eqref{a150} implies that
\[
\lb \frac{d}{dt}-\Delta\rb\tr H_t\leq C_1-C_2 \ \!\tr H_t.
\]
Hence, by the maximum principle,  we conclude  that $\tr H_t$ is bounded from above uniformly.

Next, we let $\wt H_t=\hat h h_t^{-1}$. By using a similar argument as above, we obtain that
\begin{eqnarray}\label{a160}
\lb \frac{d}{dt}-\Delta\rb\tr \wt H_t\leq\tr \lb \wt H_t \lb\Lambda_\omega R^{\hat h}-\lambda I_E\rb\rb-\tr\lb \wt H_t\xi
 h_t^{-1}\rb.
\end{eqnarray}
Since $\wt H_t$ is positive,  if we let $C_3=\sup_X \lambda_{max}\lb \Lambda_\omega R^{\hat h}-\lambda I_E\rb>0$ and pick $C_4>0$ such that $\xi\geq C_4 \hat h$, then \eqref{a160} implies that
\begin{align*}
\begin{split}
\lb \frac{d}{dt}-\Delta\rb\tr \wt H_t & \leq C_3 \tr \wt H_t-C_4 \tr\lb \wt H_t^2\rb \\
& \leq C_3 \tr \wt H_t-\frac{C_4}{r} \lb \tr \wt H_t\rb^2.
\end{split}
\end{align*}
Again, by the maximum principle,  we see that $\tr\wt H_t$ is uniformly bounded. These imply that $h_t$ is uniformly equivalent to $\hat h$ which in turn implies that there exists a constant $C_5>0$ such that
\begin{eqnarray}\label{a170}
C_5^{-1}h_0\leq h_t\leq C_5 h_0.
\end{eqnarray}

\medskip

\noindent $\bullet$ \textbf{The higher order estimates.}

\smallskip
Note that, since $K_t$ is $h_t$-self adjoint, $\Vert K_t\Vert_{h_t}^2=\tr\lb K_t^2\rb$. A simple computation shows that
\begin{eqnarray}\label{a200}
\frac{d \tr\lb K_t^2\rb}{dt}-\Delta \tr\lb K_t^2\rb=-2\tr \lb \xi h_t^{-1}K_t^2\rb-2 \Vert \nabla^{h_t}K_t \Vert_{h_t}^2\leq 0.
\end{eqnarray}
It follows from the maximum principle that $\sup_X \Vert K_t\Vert_{h_t}$ is decreasing in $t$. In particular,
\smallskip
\begin{eqnarray}\label{a300}
\sup_X \Vert K_t\Vert_{h_t}\leq \sup_X \Vert K_0\Vert_{h_0}.
\end{eqnarray}
Since
\[
\Lambda_\omega R^{h_t}=K_t+\lambda I_E+ \xi  h_t^{-1}
\]
the above estimates imply that the right hand side of this identity is uniformly bounded.
By multiplying $h_t$ to both sides of the above identity and using the $C^0$ estimate, we get the uniform $W^{2,p}$ bound for any $p>0$. Thus, the uniform $C^{1, \ga}$ estimate follows from the Sobolev embedding as we take $p$ sufficiently large.
Moreover, the $C^{k,\ga}$ estimates of $h_t$ follow from the standard Schauder theory.

Finally, if we differentiate \eqref{a100} with respect to $t$ repeatedly and by using the above uniform $C^{k,\ga}$ estimates, then we get uniform estimates of $h_t$ with respect to $t$. This leads to the long time existence of the flow \eqref{a100}.

\medskip
\noindent $\bullet$\textbf{ The exponential convergence.}

\smallskip
By \eqref{a200}, we have
\[
\frac{d \tr\lb K_t^2\rb}{dt}-\Delta \tr\lb K_t^2\rb\leq -2\tr \lb \xi h_t^{-1}K_t^2\rb.
\]
Let $C_6=\inf_X \lambda_{min}\lb \xi h_0^{-1}\rb>0$. Then, \eqref{a170} implies that
\[\xi h_t^{-1}\geq \frac{C_6}{C_5}I_E.
\]
It then follows that
\[
\frac{d \tr\lb K_t^2\rb}{dt}-\Delta \tr\lb K_t^2\rb\leq -\frac{2C_6}{C_5}\tr\lb K_t^2\rb.
\]
Namely,
\[
\lb \frac{d}{dt}-\Delta\rb \Vert K_t\Vert_{h_t}^2\leq -\gd \Vert K_t\Vert_{h_t}^2
\]
where $\gd=\frac{2C_6}{C_5}$. Thus, the maximum principle implies
\[
\sup_X \Vert K_t\Vert_{h_t}^2\leq e^{-\gd t}\sup_X \Vert K_0\Vert_{h_0}^2:=M e^{-\gd t}.
\]
On the other hand, we have
\[
\left \Vert \frac{dh_t}{dt}\right \Vert_{h_t}^2 =\tr \lb \frac{dh_t}{dt}h_t^{-1}\rb^2=\Vert K_t\Vert_{h_t}^2,
\]
which implies that, for $0\leq t_1<t_2$,
\[
d\lb h_{t_1}, h_{t_2}\rb\leq \int_{t_1}^{t_2} \left \Vert \frac{dh_t}{dt}\right \Vert_{h_t}dt\leq
\frac{2\sqrt M}{\gd}\lb e^{-\frac{\gd t_1}{2}}-e^{-\frac{\gd t_2}{2}}\rb.
\]
This implies that $h_t\to h_\infty$ exponentially. Therefore, by the uniform $C^0$ and $C^{k,\ga}$ estimates,  we know that $h_\infty$ is nondegenerate and the convergence is smooth. Furthermore,
\[
\Lambda_\omega R^{h_\infty}-\lambda I_E-\xi h_{\infty}^{-1}=K_\infty=\lim_{t\to\infty} K_t=0.
\]
Namely, $\Lambda_\omega  R^{h_\infty}=\lambda I_E+\xi h_{\infty}^{-1}$.

\end{proof}

\subsection{The Higgs Bundle Case}

We now generalize Theorem \ref{05} by allowing Higgs fields. In fact, our proof of Theorem \ref{a105} can be adapted to this case.

Let $(E, \Phi)$ be a Higgs bundle. Define 

\[\lambda_{\min}^\Phi=\inf_{\mathcal F}\mu\lb E/\mathcal F\rb,
\] where the infimum is taken over all proper $\Phi$-invariant saturated torsion-free subsheaves of $E$.

\begin{theorem}\label{b200}
We have
\[
\lambda_{\min}^\Phi=\sup_{\Lambda_\omega \lb R^h+\lsb \Phi,\Phi^{*,h}\rsb\rb>\lambda I_E} \lambda
\]
where the supremum is taken over all Hermitian metrics on $E$. 

Furthermore, for any $\lambda<\lambda_{\min}^{\Phi}$ and Hermitian metric $\xi$ on $E$, the flow

\begin{eqnarray}\label{b100}
\begin{cases}
\frac{\pa h_t}{\pa t}h_t^{-1}= -\lb \Lambda_\omega\lb R^{h_t}+\lsb \Phi,\Phi^{*,h_t}\rsb\rb-\lambda I_E-\xi h_t^{-1}\rb\\
h(0)=h_0.
\end{cases}
\end{eqnarray}
exists for all time and converges smoothly to the unique solution of \[ \Lambda_\omega \lb R^h+\lsb \Phi,\Phi^{*,h}\rsb\rb=\lambda I_E+\xi h^{-1}.\]
\end{theorem}

\begin{proof}  [Sketch of Proof]

Similar to Theorem \ref{a105}, we only need to show the convergence of flow \eqref{b100}. We first note that, for any Hermitian metrics $\zeta, \eta$ on $E$, one has the inequality
\[
\tr \lb \zeta \eta^{-1}\lb \Lambda_\omega \lsb \Phi,\Phi^{*,\zeta}\rsb -\Lambda_\omega \lsb \Phi,\Phi^{*,\eta}\rsb\rb \rb\geq 0.
\]
For any Hermitian metric $h$, we let $F^h:=R^h+\lsb \Phi,\Phi^{*,h}\rsb-\lambda I_E$. Again, we fix $\hat h$ on $E$ such that \[\Lambda_\omega \lb R^{\hat h}+\lsb \Phi,\Phi^{*,\hat h}\rsb\rb-\lambda I_E>0.\]
Similarly, we let $H_t=h_t\hat h^{-1}$, where $h_t$ is the solution to equation \eqref{b100}, then, by inequality \eqref{a140},  we have
\begin{align*}
\begin{split}
\Delta \lb \tr H_t \rb \geq & \tr \lb H_t \Lambda_\omega F^{\hat h}\rb-\tr \lb H_t \Lambda_\omega F^{h_t}\rb +
\tr \lb H_t\lb \Lambda_\omega \lsb \Phi,\Phi^{*,h_t}\rsb -\Lambda_\omega \lsb \Phi,\Phi^{*,\hat h}\rsb\rb \rb\\
\geq & \tr \lb H_t \Lambda_\omega F^{\hat h}\rb-\tr \lb H_t \Lambda_\omega F^{h_t}\rb.
\end{split}
\end{align*}
The essentially same argument as in the proof of Theorem \ref{a105} shows that $\tr H_t$ is uniformly bounded. Similar arguments also show that $\tr \lb \hat h h_t^{-1}\rb$ is uniformly bounded and hence we obtain the uniform $C^0$ estimate.

To derive the higher order estimates,  we let $K_t=\Lambda_\omega F^{h_t}-\xi h_t^{-1}$ and notice that, parallel to formula \eqref{a200}, we have
\[
\frac{d \tr\lb K_t^2\rb}{dt}-\Delta \tr\lb K_t^2\rb=-2\tr \lb \xi h_t^{-1}K_t^2\rb-2 \Vert \nabla^{h_t}K_t \Vert_{h_t}^2-2 \Vert \lsb \Phi,K_t\rsb \Vert_{h_t}^2\leq 0.
\]
Also, we know that
\[
\Lambda_\omega R^{h_t}=K_t-\lsb\Phi,\Phi^{*,h_t}\rsb +\lambda I_E +\xi h_t^{-1}.
\]
By using the $C^0$ estimate, it follows that the right side of the above is uniformly bounded.
The rest of the proof is essentially the same as that of Theorem \ref{a105}.

\end{proof}


\begin{thebibliography}{99}



\bibitem {B et al} R. Berman, S. Boucksom, P. Eyssidieux, V. Guedj, and A. Zeriahi. K\"ahler-Einstein metrics and the K\"ahler-Ricci flow on log Fano varieties, {\it J. Reine Angew. Math.} 751 (2019), 27–89.


\bibitem{sd0}
S. K. Donaldson. A new proof of a theorem of Narasimhan and Seshadri. {\it J. Differential Geom.} 18 (1983), no. 2, 269–277.

\bibitem{sd1}
S. K. Donaldson. Anti self-dual Yang-Mills connections over complex algebraic surfaces and stable vector bundles. {\it Proc. London Math. Soc. (3)}, 50 (1):1--26, 1985.



\bibitem {sd2}
S. K. Donaldson. Infinite determinants, stable bundles and curvature. {\it Duke Math. J.}, 54 (1):231--247, 1987.

\bibitem{DR}
T. Darvas and Y. Rubinstein. Convergence of the K\"ahler-Ricci iteration, {\it Anal. PDE} 12 (3) (2019), 721-735.

\bibitem{WYY2}
J. Fan, M. Wang, X. Yang, and S.-T. Yau. Existence of Hermitian metrics with prescribed Hermitian- Yang-Mills Tensors II, arXiv: 2604.02679 (2026).


\bibitem{jk}
J. Keller. Ricci iterations on K\"ahler classes, {\it J. Inst. Math. Jussieu} 8:4 (2009), 743--768.

\bibitem{ko}
S. Kobayashi. Differential geometry of complex vector bundles, {\it Princeton University Press,} Princeton, NJ, 2014, xi + 304 pp.

\bibitem{LY} J. Li and S.-T. Yau. Hermitian-Yang-Mills connection on non-K\"ahler manifolds. {\it Mathematical aspects of string theory (San Diego, Calif., 1986)}, 560--573, Adv. Ser. Math. Phys., 1, {\it World Sci. Publishing, Singapore}, 1987


\bibitem{na}
A. Nadel. On the absence of periodic points for the Ricci curvature operator acting on the space of K\"ahler metrics. Modern methods in complex analysis (Princeton, NJ, 1992), 277–281,
Ann. of Math. Stud., 137, Princeton Univ. Press, Princeton, NJ, 1995.


\bibitem{yr}
Y. A. Rubinstein. Some discretizations of geometric evolution equations and the Ricci iteration on the space
of K\"ahler metrics, {\it Adv. Math. 218:5 (2008),} 1526–1565.

\bibitem {sim1988}
C. T. Simpson. Constructing variations of Hodge structure using Yang-Mills theory and applications to uniformization. {\it J. Amer. Math. Soc.}, 1 (4):867--918, 1988.


\bibitem{UY}
K. Uhlenbeck and S.-T. Yau. On the existence of Hermitian-Yang-Mills connections in stable vector bundles. volume 39, pages S257–S293. 1986. Frontiers of the mathematical sciences: 1985 (New York, 1985).

\bibitem{WYY1}
M. Wang, X. Yang, and S.-T. Yau. Existence of Hermitian metrics with prescribed Hermitian- Yang-Mills Tensors I, arXiv: 2603.10611 (2026).




\bibitem{WYY3}
M.~Wang, X.~Yang, and S.-T. Yau.
\newblock {E}xistence of {T}wisted {H}ermitian-{E}instein {M}etrics on
  {U}nstable {V}ector {B}undles.
\newblock arXiv:2606.15102 (2026).

\bibitem{yau}
S.-T. Yau. On the Ricci curvature of a compact Kähler manifold and the complex Monge-Amp\'ere equation. I. {\it Comm. Pure Appl. Math.} 31 (1978), no. 3, 339–411.

\bibitem{kz}
K. Zhang. The Ricci iteration towards cscK metrics, {\it Adv. Math. 475} (2025) 110340.

\end{thebibliography}

\end{document}